\documentclass[oneside,reqno,12pt]{amsart}
\usepackage[greek,english]{babel}
\usepackage{amsthm}
\usepackage{amsbsy}
\usepackage{amsfonts}
\usepackage{graphicx}
\usepackage{hyperref}
\hypersetup{colorlinks=true, citecolor=blue, linkcolor=red}

\newtheorem{thm}{Theorem}

\newtheorem{rem}{Remark}

\begin{document}
\title[Minimality  of the connecting solution of the Painlev\'{e}-II PDE]{Energy minimality property of the connecting solution of the Painlev\'{e} phase transition model}
\author{Christos Sourdis}
\address{University of Athens, Greece.
}
\email{sourdis@uoc.gr}

\date{\today}
\begin{abstract}
 We establish the energy minimality property of solutions to the generalized  Painlev\'{e}-II equation $\Delta y-x_1y-2y^3=0$, $(x_1,x_2)\in \mathbb{R}^2$, which are increasing in $x_2$ and converge to the positive and negative Hastings-McLeod solutions as $x_2\to \pm \infty$.
 \end{abstract}
 \maketitle

%\section{Introduction}
The  Painlev\'{e}-II equation
\begin{equation}\label{eqPain1}
y''-xy-2y^3=0, \ x\in \mathbb{R},
\end{equation}
admits a unique positive solution $Y$ which is called the Hastings-McLeod solution (see \cite{hastings}). We note that $Y'<0$ and
\[
Y(x)-\sqrt{\frac{-x}{2}}   \to 0 \ \textrm{as}\ x\to -\infty,\ \ Y(x)\to 0 \ \textrm{as}\ x\to +\infty.
\]
In fact, as was shown in \cite{clercCAlc}, the Hastings-McLeod solution is an energy minimizer in the sense of Morse (defined analogously to (\ref{eqMorse}) below).
In this regard, we note that the Hastings-McLeod solution appears in the blow-up analysis of corner layer singularities of minimizers of a class of singularly perturbed Gross-Pitaevskii energies (see \cite{clercCAlc,karal} and the references therein).

Special solutions of the generalized Painlev\'{e}-II equation
\begin{equation}\label{eqEq}
\Delta y-x_1y-2y^3=0,\ \ (x_1,x_2)\in \mathbb{R}^2,
\end{equation}
are in the core of more complicated singularities of critical points of such Gross-Pitaevskii energies, where corner layers and transition layers interact around a junction point (see \cite{kowalArxivia}).
Again the interest is in solutions of (\ref{eqEq}) that are energy minimizers in the sense of Morse, that is
\begin{equation}\label{eqMorse}
E(y,\textrm{supp}\phi)\leq E(y+\phi,\textrm{supp}\phi)\ \ \ \forall \ \phi\in C_0^\infty(\mathbb{R}^2),
\end{equation}
where we  have denoted
\[
E(u,\Omega)=\int_{\Omega}^{}\left[\frac{1}{2}|\nabla u|^2+
\frac{1}{2}
x_1u^2 +
\frac{1}{2}
u^4\right]dx_1dx_2
\]
for any bounded subset $\Omega$ of $\mathbb{R}^2$. As we already pointed out,  the Hastings-McLeod solution $Y$ is energy minimal for (\ref{eqPain1}), and thus it is also for (\ref{eqEq}) (this can be proven as in \cite[Lem. 3.2]{jerisson}, see also Remark \ref{rem} herein for a simple proof). The current paper is motivated by the natural question whether (\ref{eqEq}) admits energy minimizing solutions that depend (nontrivially) on both  coordinates.

In the recent paper \cite{kowalArxivia} a solution to (\ref{eqEq}) was constructed as an energy minimizer with respect to compactly supported perturbations $\phi$ in (\ref{eqMorse}) that are odd in $x_2$. In particular,  $y$ is odd with respect to $x_2$ and can be chosen such that $y_{x_1}<0$ for $x_2>0$,
\begin{equation}\label{eq1}
y_{x_2}>0,
\end{equation}
\begin{equation}\label{eq2}
y(x_1,x_2+l)\to \pm Y(x_1)\ \textrm{in}\ C^2_{loc}(\mathbb{R}^2)\ \textrm{as}\ l\to \pm\infty,
\end{equation}
\[
y(x_1,x_2)\to 0\ \textrm{as}\ x_1\to +\infty \ (\textrm{uniformly in}\  x_2),
\]
while as $x_1\to -\infty$ it connects to a suitable singular scaling of the hyperbolic tangent kink solution to the Allen-Cahn equation (\ref{eqAC}).
It is worth mentioning that such a solution was previously briefly discussed in \cite{karal} (see relation (211) therein).
 We note in passing that  the division trick (\ref{eqwdef}) below, which was not used in \cite{kowalArxivia}, may clarify further the interesting link with the Allen-Cahn equation.

It is expected that the above solution  to (\ref{eqEq}) is an energy minimizer with respect to arbitrary compactly supported perturbations (not necessarily odd in $x_2$), that is (\ref{eqMorse}) holds (see also \cite{kowalBirs,smyrnelis}). This expectation is further supported by the fact that (\ref{eq1}) implies that $y$ is a linearly stable solution of (\ref{eqEq}) (see \cite{cabreLNM},\cite[Ch. 1.2]{dupaigne} for the appropriate definition in this context). If it is indeed the case, then this solution $y$  would be a natural two-dimensional analog of the Hastings-McLeod solution. Moreover, it would provide an example of an energy minimizing heteroclinic connection between nontrivial solutions to a scalar PDE in low dimensions (see, however, \cite{liu} for the case of the Allen-Cahn equation (\ref{eqAC}) in high dimensions). To the best of our knowledge, heteroclinic solutions with these properties have only been shown to exist in the case of vector Allen-Cahn systems in the plane (see \cite{s}).

The importance of showing the energy minimality of the solution $y$ can also be highlighted by the following observation. In the previously mentioned related problems, it is the case that the existence of an antisymmetric energy minimizing solution implies the existence of an energy minimizing heteroclinic connection between nontrivial solutions in one dimension higher (see \cite{jerisson}).

In this short paper we verify the aforementioned expectation.
In light of the above discussion, our result is expected to play an important role in the De Giorgi type program for (\ref{eqEq}) (in $\mathbb{R}^N$) that was proposed recently in \cite{kowalArxivia}. More precisely, we establish the following result.

\begin{thm}\label{thm}
Let $y$ be a solution of (\ref{eqEq}) that satisfies (\ref{eq1}) and (\ref{eq2}). Then, it satisfies the energy minimality property (\ref{eqMorse}).
\end{thm}
\begin{proof}
Our proof is motivated by a well known division trick from the study of Ginzburg-Landau energies due to \cite{miro}. The main idea is to consider  the quotient
\begin{equation}\label{eqwdef}
w=\frac{y}{Y},
\end{equation}
which satisfies
\begin{equation}\label{eqweq}
  \textrm{div}(Y^2\nabla w)+2Y^4(w-w^3)=0,\ \ (x_1,x_2)\in \mathbb{R}^2.
\end{equation}
Moreover, we observe that (\ref{eq1}) and (\ref{eq2}) imply that
\begin{equation}\label{eq1=}
  w_{x_2}>0
\end{equation}
and
\begin{equation}\label{eq2=}
w(x_1,x_2+l)\to \pm 1\ \textrm{in}\ C^2_{loc}(\mathbb{R}^2)\ \textrm{as}\ l\to \pm\infty,
\end{equation}respectively.
In turn, we obtain that
\begin{equation}\label{eqmaxw}
|w(x_1,x_2)|=\frac{|y(x_1,x_2)|}{|Y(x_1)|}<1,\ \ (x_1,x_2)\in \mathbb{R}^2.
\end{equation}

We can then exploit the ideas from a  proof of a related result for the Allen-Cahn equation
\begin{equation}\label{eqAC}
  \Delta u+u-u^3=0 \ \textrm{in}\ \mathbb{R}^N, \ N\geq 1,
\end{equation}
which establishes the energy minimality of solutions such that $u_{x_N}>0$, $|u|<1$ and $u\to \pm 1$  as $x_N\to \pm \infty$.
This result was first proven in   \cite{alberti} by constructing a calibration, while a simpler proof  based on the sliding method and the maximum principle is presented in \cite[Thm. 1.32]{cabreLNM}. In the remainder of the proof we will adapt the proof in the latter reference and use crucially that (\ref{eqEq}) is translation invariant in the $x_2$ direction.

To establish the energy minimality of $y$, without loss of generality, it is enough to show that for any ball $B_R\subset \mathbb{R}^2$ the solution $y$ minimizes $E(\cdot,B_R)$ in the set
$
\left\{y+W_0^{1,2}(B_R)\right\}$.
The direct method of the calculus of variations can be applied to establish the existence of a minimizer $z$ of $E(\cdot,B_R)$ in the set
$
\left\{y+W_0^{1,2}(B_R)\right\}
$. The function $z$ satisfies in the classical sense
\[
\left\{\begin{array}{ll}
         \Delta z-x_1z-2z^3=0 &\textrm{in}\ B_R, \\
          |z|<Y & \textrm{in}\ \overline{B_R},\\
        z=y & \textrm{on}\ \partial B_R.
       \end{array}
 \right.
\]
We point out that the second property follows as in \cite[Lem. 3.1]{jerisson} from (\ref{eqmaxw}) and the energy minimality of the Hastings-McLeod solution $Y$ and of $z$.
Our goal is to show that $z\equiv y$ in $\overline{B_R}$, which would clearly imply the assertion of the theorem.

The quotient
\[
\omega=\frac{z}{Y}
\]
satisfies
\begin{equation}\label{eqomegaEq}
\left\{\begin{array}{ll}
         \textrm{div}(Y^2\nabla \omega)+2Y^4(\omega-\omega^3)=0 &\textrm{in}\ B_R, \\
          |\omega|<1 & \textrm{in}\ \overline{B_R},\\
        \omega=w & \textrm{on}\ \partial B_R.
       \end{array}
 \right.
\end{equation}

Let us consider the functions
\[
w^t(x_1,x_2)=w(x_1,x_2+t)\ \textrm{for}\ t\in \mathbb{R}.
\]
Clearly, these still satisfy (\ref{eqweq}).
Moreover, by the monotonicity assumption (\ref{eq1=}), we have that
\begin{equation}\label{eqmono}w^t < w^{t'}	
\ \textrm{in}\  \mathbb{R}^2\  \textrm{if}\ t<t'.\end{equation}
So, in light of (\ref{eq2=}),  the graphs of $w^t$, $t\in \mathbb{R}$,  form a foliation
that fills all of $\mathbb{R}^2 \times (-1, 1)$.

By (\ref{eq2=}) and the second relation in (\ref{eqomegaEq}), we have that the graph of $w^t$ in the compact set $\overline{B_R}$ is above the graph
of $\omega$ for $t$ large enough, and it is below the graph of $\omega$ for $t$ negative enough. If $w\not\equiv \omega$, we can assume without loss of generlity that
$\omega<w$ at some point in $B_R$.  It follows that, starting from $t =-\infty $, there
will exist a first $t_* < 0$ such that $w^{t_*}$ touches $\omega$ at a point $P\in \overline{B_R}$. This means that
$w^{t_*} \leq \omega$ in $\overline{B_R}$ and $w^{t_*} (P ) = \omega(P )$.

By (\ref{eqmono}), recalling that $t_* < 0$ and   that $\omega = w = w^0$ on $\partial B_R$, we see that the point $P$ cannot
belong to $\partial B_R$. Thus, $P$ has to  be an interior point of $B_R$.
But then we have that $w^{t_*}$ and $\omega$ are two solutions of the same semilinear equation  (\ref{eqweq}), the graph of $w^{t_*}$ stays below that of $\omega$, and they touch
each other at the interior point $\left(P , \omega(P )\right)$. This is a contradiction with the strong
maximum principle. Hence, we conclude that the desired relation $w\equiv \omega$ holds. Consequently, we deduce that $y\equiv z$ which completes the proof.
\end{proof}
\begin{rem}\label{rem}
A simple argument for showing the energy minimality of the Hastings-McLeod solution $Y$ with respect to (\ref{eqEq}) (in any dimension) can be given by a foliation argument as in the above proof. To this end, we note that the family $\left\{tY\ :\ t\geq 1 \right\}$ provides a foliation by supersolutions to (\ref{eqEq}) of the set  $\left\{(x_1,x_2,v)\right.$ $\left. \ :\ v\geq Y(x_1)\right\}$, while the family $\left\{tY\ :\ 0\leq t\leq 1 \right\}$ provides a foliation by lower solutions to (\ref{eqEq}) of the set  $\left\{(x_1,x_2,v)\right.$ $\left. \ :\ 0\leq v\leq Y(x_1)\right\}$.
\end{rem}

\end{document}